\documentclass[10pt,twocolumn]{amsart}
\usepackage{pdfsync}

\usepackage{amsmath, amssymb}
\usepackage{tikz} 

\usetikzlibrary{decorations.pathmorphing}

\newcommand{\EMPTY}{\varnothing}

\newcommand{\bbOS}{\mathbb {OS}}

\DeclareMathOperator{\Iso}{Iso}
\DeclareMathOperator{\Met}{{\mathcal M}}

\newcommand{\bbN}{\mathbb N}
\newcommand{\bbQ}{{\mathbb Q}}

\newcommand{\bbE}{\mathbb E}
\newcommand{\bbF}{\mathbb F}
\newcommand{\bbX}{\mathbb X}
\newcommand{\bbR}{\mathbb R}
\newcommand{\bbC}{\mathbb C}

\newcommand{\bbK}{\mathbb K}
\newcommand{\bbU}{\mathbb U}
\newcommand{\bbY}{\mathbb Y}

\newcommand{\e}{\varepsilon}

\newtheorem{theorem}{Theorem}[section]

\newtheorem{question}[theorem]{Question}

\newtheorem{lemma}[theorem]{Lemma}

\newtheorem{appthm}{Theorem}

\theoremstyle{definition}

\newtheorem{definition}[theorem]{Definition}

\newcounter{my_enumerate_counter}
\newcommand{\pushcounter}{\setcounter{my_enumerate_counter}{\value{enumi}}}
\newcommand{\popcounter}{\setcounter{enumi}{\value{my_enumerate_counter}}}

\DeclareMathOperator{\dom}{dom}

\DeclareMathOperator{\Sym}{Sym}
\DeclareMathOperator{\cl}{cl}

\newcommand{\bbZ}{\mathbb Z}

\DeclareMathOperator{\Aut}{Aut}

\DeclareMathOperator{\graph}{graph}
\DeclareMathOperator{\proj}{proj}

\newcommand{\restrict}{\!\upharpoonright\!}

\address{Department of Mathematics, University of Toronto, Toronto, Ontario M5S 2E4, Canada}
\email{elliott@math.toronto.edu}

\address{Department of Mathematics and Statistics, York University, 4700 Keele St., Toronto, Ontario, M3J 1P3, Canada}
\email{ifarah@mathstat.yorku.ca}

\address{Department of Mathematics, University of Houston, Houston, TX 77204-3476}
\email{vern@math.uh.edu}

\address{Department of Mathematics, Statistics, and Computer Science (M/C 249), University of Illinois at Chicago, 851 S. Morgan St., Chicago, IL 60607-7045, USA}
\email{rosendal.math@gmail.com}

\address{Department of Mathematics, Purdue University, 150 N. University St., West Lafayette IN 47906, USA}
\email{atoms@purdue.edu}

\address{Department of Mathematical Sciences, University of Copenhagen, Universitetsparken 5, 2100 Copenhagen, Denmark}
\email{asgert@math.ku.dk}

\title{The isomorphism relation for separable C*-algebras}

\author[Elliott et al.]{George A. Elliott, Ilijas Farah, Vern Paulsen, Christian Rosendal, Andrew S.  Toms and Asger T\"ornquist}

\date{\today}
\begin{document}

\begin{abstract}
We prove that the isomorphism relation for separable C$^*$-algebras, and also  the relations of complete and $n$-isometry for operator spaces and systems, are Borel reducible to the orbit equivalence relation of a Polish group action on a standard Borel space.
\end{abstract}

\maketitle

\section{Introduction}

The problem of classifying a collection of objects up to some notion of isomorphism can usually be couched as the study of an analytic equivalence relation on a standard Borel space parameterizing the said objects.  Such relations admit a notion of comparison, {\it Borel reducibility}, which allows one to assign a degree of complexity to the classification problem.  If $X$ and $Y$ are standard Borel spaces admitting equivalence relations $E$ and $F$ respectively then we say that $E$ is {\it Borel reducible} to $F$, written $E \leq_B F$, if there is a Borel map $\Theta:X \to Y$ such that
\[
xEy \iff \Theta(x)F\Theta(y).
\]
In other words, $\Theta$ carries equivalence classes to equivalence classes injectively.  We view $E$ as being ``less complicated" than $F$. There are some particularly prominent degrees of complexity in this theory which serve as landmarks for classification problems in general.  For instance, a relation $E$ is {\it classifiable by countable structures} (CCS) if it is Borel reducible to the isomorphism relation for countable graphs.  Classification problems in functional analysis (our interest here) tend not to be CCS, but may nevertheless be ``not too complicated" in that they are Borel reducible to the orbit equivalence relation of a Borel action of a Polish group on a standard Borel space;  this property is known as being {\it below a group action}.

The connections between Borel reducibility and operator algebras have received considerable attention lately.  Using Hjorth's theory of turbulence developed in \cite{Hj:Book}, it has been shown that several classes of operator algebras
are not CCS. This applies to von Neumann factors of every type \cite{sato09a}, 
ITPFI$_2$ factors \cite{SaTo:Turbulence}, 
separable simple unital nuclear C$^*$-algebras \cite{FaToTo:Turbulence}, 
and spaces of irreducible representations of non-type I C$^*$-algebras, \cite{KerLiPi:Turbulence,Fa:Dichotomy}. On the other hand, Elliott's K-theoretic classification of AF algebras together with the Borel computability of K-theory \cite{FaToTo:Descriptive} show that AF algebras are CCS.

We are interested in classifying separable C$^*$-algebras, operator spaces, and operator systems.  The classification problem for nuclear simple separable C$^*$-algebras was studied in \cite{FaToTo:Turbulence}, where the isomorphism relation was shown to be below a group action.  Establishing this upper bound was rather involved;  it required Borel versions of Kirchberg's $\mathcal{O}_2$ embedding and absorption theorems.  It of course invited the question of whether isomorphism of \emph{all} separable C$^*$-algebras is below a group action.  Here we give a surprisingly simple proof of the following result.

\begin{theorem}\label{T1} 
Each of the following equivalence relations 
is below a group action: 
\begin{enumerate}
\item  The isomorphism relation for separable C$^*$\! -\! algebras.
\item Complete isometry, as well as $n$-isometry for any $n$, 
 of separable operator spaces.
\item Complete isometry, as well as $n$-isometry for any $n$, 
 of separable operator systems. 
\item Unital, complete isometry of separable operator systems.
\item  Unital, complete order isomorphism of separable operator systems.
 \end{enumerate}
\end{theorem}
\noindent Part (1)  of course improves on one of the main results
of \cite{FaToTo:Turbulence}. Parts (2--5) give a partial answer to a question stated by Ed Effros during the January 2012 meeting ``Set theory and C*-algebras" at the American Institute of Mathematics. Theorem \ref{T1} follows from a more general result proved in \S\ref{s.PS}. We show that $\Iso(\bbU)$, the isometry group of the Urysohn 
metric space $(\bbU,\delta)$, plays a role for the isometry relation  of what we term {\it Polish structures in a countable signature} (see \cite{BYBHU,FaHaSh:Model2}), analogous to the role played by the Polish group $S_\infty$ of all permutations of $\mathbb{N}$ for the isomorphism relation of countable structures.

\section{Polish structures}\label{s.PS}

The main ingredients for proving Theorem \ref{T1} are the notion of a \emph{Polish structure}, which we introduce now, and Theorem \ref{t.bga} below.

\begin{definition}\label{d.polstruct}
Let $\mathcal L=(l_1,\ldots)$ be a finite or infinite sequence in $\bbN$. A \emph{Polish $\mathcal L$-structure}\footnote{We consciously avoid the term ``metric structure'' since this is already used for a slightly different notion in continuous logic.
} is a triple 
\[
\bbX=(X,d^{\bbX},(F_n^{\bbX})),
\]
 where $(X,d)$ is a separable complete metric space, called the \emph{domain}, and $F_n^{\bbX}\subseteq X^{l_n}$ is closed in the product topology. We think of $F_n^{\bbX}$ as a relation on $X$ and call $l_n$ the \emph{arity} of $F_n^{\bbX}$. The sequence $\mathcal L$ is called the \emph{signature} or \emph{arity sequence} of the structure $\bbX$. (Below we will suppress the superscript $\bbX$ whenever possible.)

Two Polish structures $\bbX$ and $\bbY$ with the same signature $\mathcal L$ are said to be \emph{isometrically isomorphic} if there is an isometric bijection $h: X\to Y$ such that for all $n$ and $(x_1,\ldots,x_{l_n})\in X^{l_n}$ we have
$$
(x_1,\ldots,x_{l_n})\in F_n^{\bbX}\iff (h(x_1),\ldots, h(x_{l_n}))\in F_n^{\bbY}.
$$
\end{definition}

Let $(Y,d)$ be a metric space, let $\mathcal L=(l_1,\ldots)$, and define $M(\mathcal L,Y,d)\subseteq (F(Y)\setminus\{\EMPTY\})\times \prod_{n} F(Y^{l_n})$ by
$$
M(\mathcal L,Y,d)=\{(X,(F_n)): F_n\subseteq X^{l_n}\}.
$$
This set is Borel (by \cite[12.11]{Ke:Classical}), and to each $\bbX=(X,(F_n))\in M(\mathcal L,Y,d)$ we have the Polish $\mathcal L$-structure $(X,d\restrict X,(F_n))$, which we also denote by $\bbX$. For $\bbX, \bbY\in M(\mathcal L,Y,d)$, write $\bbX\simeq^{M(\mathcal L,Y,d)}\bbY$ if $\bbX$ and $\bbY$ are isomorphic Polish $\mathcal L$-structures.

When $(Y,d)$ is the Urysohn metric space $\bbU$, we will usually write $M(\mathcal L)$ rather than $M(\mathcal L,\bbU,\delta)$. This is motivated by the fact that every separable metric space can be isometrically embedded into $\bbU$, and so every Polish $\mathcal L$-structure is isomorphic to some $\bbX\in M(\mathcal L)$. Thus $M(\mathcal L)$ provides a parametrization of all Polish $\mathcal L$-structures as a standard Borel space. It is not hard to see that that this moreover is a good standard Borel parametrization in the sense of \cite{FaToTo:Turbulence}.

The group $\Iso(\bbU)$ acts diagonally on $\bbU^{l_i}$ for each $i$, and so it acts naturally on $M(\mathcal L)$ by $\sigma\cdot (X,(F_n))=(\sigma\cdot X,(\sigma\cdot F_n))$. This induces the orbit equivalence relation 
$$
\mathcal X\equiv^{M(\mathcal L)}\mathcal Y\iff (\exists\sigma\in\Iso(\bbU))\,\sigma\cdot\mathcal X=\mathcal Y.
$$
The following theorem was inspired by   \cite[\S 2]{GaKe}.

\begin{theorem}\label{t.bga}
We have $\simeq^{M(\mathcal L)}\,\leq_B\,\equiv^{M(\mathcal L)}$, and so $\simeq^{M(\mathcal L)}$ is below a group action.
\end{theorem}

To prove this, we first need a uniformly Borel version of injective universality and homogeneity of $\bbU$. This will also be used several times later. For a countable set $A$, define
$$
\mathcal M_A=\{d\in \bbR^{A\times A}: d\text{ is a metric on }A\}.
$$
This is easily seen to form a $G_\delta$ subset of $\bbR^{A\times A}$, whence $\mathcal M_A$ is Polish in the subspace topology. When $A=\bbN$ we let $\mathcal M=\mathcal M_\bbN$.

\begin{lemma}\label{l.unifhom}
There are Borel functions $\vartheta_n:\mathcal M\to\bbU$, $n\in\bbZ$, such that for all $d,d'\in\mathcal M$:
\begin{enumerate}
\item the set $\{\vartheta_n(d):n\in\bbZ\}$ is dense in $\bbU$;
\item $\delta(\vartheta_n(d),\vartheta_m(d))=d(n,m)$ for all $n,m\in\bbN$;
\item any isometric bijection 
$$
\sigma:\overline{\{\vartheta_n(d):n\in\bbN\}}\to\overline{\{\vartheta_n(d'):n\in\bbN\}}
$$ extends to an isometric automorphism of $\bbU$. 
\end{enumerate}
\end{lemma}

\begin{proof}
We adapt the Katetov construction of the Urysohn metric space; see \cite{Kat} and \cite{GaKe}. Fix an enumeration $(A_n)_{n\in\bbN}$ of all finite non-empty subsets of $\bbZ$, and let $*$ be a point \emph{not} in $\bbZ$. For $d\in \mathcal M$, call a metric $\rho\in\mathcal M_\bbZ$ a \emph{Katetov extension} of $d$ if for all $n\in\bbN$, all $\e>0$, and every metric $\tilde d$ on $A_n\cup\{*\}$ satisfying
\begin{enumerate}
\item $\tilde d(x,y)=d(x,y)$ for all $x,y\in A_n$, and
\item $\tilde d(x,*)\in\bbQ_+$ for all $x\in A_n$
\end{enumerate}
there is is some $i\in\bbZ\setminus A_n$ such that $|\rho(i,x)-\tilde d(*,x)|<\varepsilon$ for all $x\in A_n$.
It follows directly from the definition that the set
$$
U=\{(d,\rho)\in\mathcal M\times \mathcal M_\bbZ:\rho\text{ is a Katetov extension of } d\}
$$
is $G_\delta$. The group $\Sym(\bbZ\setminus\bbN)$ of all permutations of $\bbZ\setminus\bbN$ acts naturally on $\mathcal M_\bbZ$, and the sections $U_d=\{\rho:(d,\rho)\in U\}$ are invariant under this action. The Katetov property guarantees that the $\Sym(\bbZ\setminus\bbN)$-orbits in $U_d$ are dense in $U_d$. It follows from Theorem \ref{t.hsp} (see appendix) that there is a Borel $\psi:\mathcal M\to \mathcal M_\bbZ$ such that $(d,\psi(d))\in U$ for all $d\in\mathcal M$. A standard approximate intertwining/back-and-forth argument (e.g., \cite{Ell:Towards}) shows that for any $(d,\rho)\in U$, the completion of $(\bbZ,\rho)$ is is isometrically isomorphic to $(\bbU,\delta)$, and a Borel coding of this argument analogous to \cite[Theorem 7.6]{FaToTo:Turbulence} shows that there are Borel maps $\vartheta_n:\mathcal M\to\bbU$, $n\in\bbZ$, satisfying (1) and (2) above. Finally, another approximate intertwining/back-and-forth argument can be used to establish (3). 
\end{proof}

\begin{proof}[Proof of Theorem \ref{t.bga}]
We will define a Borel function $M(\mathcal L)\to M(\mathcal L):\bbX\mapsto\bbX'$ such that $\bbX\simeq^{M(\mathcal L)}\bbY$ if and only if  $\bbX'\equiv^{M(\mathcal L)} \bbY'$, and $\bbX\simeq^{M(\mathcal L)}\bbX'$. For $\bbX,\bbY\in M(\mathcal L)$ with domains $X$ and $Y$ finite, the homogeneity of $\bbU$ immediately gives that $\bbX\simeq^{M(\mathcal L)} \bbY$ if and only if  $\bbX\equiv^{M(\mathcal L)}\bbY$, so here we may simply take $\bbX=\bbX'$.

For the case when $\bbX$ has infinite domain, recall that by the Kuratowski--Ryll-Nardzewski theorem we may find Borel functions $\psi_n^{k}: F(\bbU^k)\to\bbU^k$ such that $\{\psi_n^k(F):n\in\bbN\}$ is dense in $F$ whenever $F\neq\EMPTY$. From the $\psi_n^k$ we can define a sequence of Borel functions $\tilde\psi_n:F(\bbU)\to\bbU$ such that whenever $\bbX=(X,(F_k))\in M(\mathcal L)$ and $X$ is infinite, the sequence $\tilde\psi_n(X)$ gives an injective enumeration of a dense subset of $X$ and for all $k$, $\psi^{l_k}_n(F_k)\subseteq \{\tilde\psi_n(X):n\in\omega\}^{l_k}$. Define an element of $d_X\in\mathcal M$ by $d_X(n,m)=\delta(\tilde\psi_n(X),\tilde\psi_m(x))$, and let $\gamma_n(X)=\vartheta_n(d_X)$, $n\in\bbZ$, where the $\vartheta_n$ are provided by Lemma \ref{l.unifhom}. Then the $\gamma_n$ are Borel, and if we define $\bbX'=(X',F_n')$ by letting
$$
X'=\cl(\{\gamma_n(X):n\in\bbN\})
$$
and
\begin{align*}
F_n'=&\cl(\{(\gamma_{i_1}(X),\ldots,\gamma_{i_{l_n}}(X)):\\ 
&(\exists (i_j))(\tilde\psi_{i_1}(X),\ldots,\tilde\psi_{i_{l_n}}(X))\in F_n\}),
\end{align*}
then $\bbX'\simeq^{M(\mathcal L)}\bbX$. That the map $\bbX\mapsto \bbX'$ is Borel follows from the Kuratowski--Ryll-Nardzewski theorem, and that $\bbX\simeq^{M(\mathcal L)}\bbY$ precisely when $\bbX'\equiv^{M(\mathcal L)}\bbY'$ is immediate from Lemma \ref{l.unifhom}.(3).
\end{proof}

\section{Proof of Theorem \ref{T1}}

\subsection{C*-algebras}\label{S.C*} A C*-algebra $A$ can be cast as a Polish structure with domain $A$ as follows: We have the relations $F_0=\{0\}\subseteq A$, $F_{+},F_{\cdot}\subseteq A^3$, which are the graphs of the functions $(a,b)\mapsto a+b$ and $(a,b)\mapsto ab$, respectively, $F_{*}\subseteq A^2$ which is the graph of the map $a\mapsto a^{*}$, and for each $q\in\bbQ(i)=\bbQ+i\bbQ$ the relations 
$$
F_q=\{(a,b)\in A^2: b=qa\}.
$$
The signature is $\mathcal L_{C^*}=(1,3,3,2,2,\ldots)$. It is easy to check directly that the set $M_{C^*}\subseteq M(\mathcal L_{C^*})$ which corresponds to C*-algebras is Borel. (See also Lemma \ref{L.2} below, where a more general statement is proven.) Let $\simeq^{M_{C^*}}$ denote the isomorphism relation in $M_{C^*}$. It can easily be shown that $M_{C^*}$ provides a standard Borel parametrization of the class of separable C*-algebras, in the sense of \cite[Definition~2.1]{FaToTo:Turbulence}). Clearly $\simeq^{M_{C^*}}$ is the restriction of $\simeq^{M(\mathcal L_{C^*})}$ to $M_{C^*}$, and so by Theorem \ref{t.bga} above $\simeq^{M_{C^*}}$ is below a group action.

To finish the proof of Theorem \ref{T1}.(1) we need only show that the parametrization $M_{C^*}$ is weakly equivalent to the parametrizations given in \cite{FaToTo:Turbulence}, in the sense of \cite[Definition 2.1]{FaToTo:Turbulence}. For this, recall from \cite[\S 2.4]{FaToTo:Turbulence} that $\hat \Xi$ is the space of countable normed $\bbQ(i)$-$*$-algebras with domain $\bbN$ which satisfy the C*-axiom. For each $\xi\in\hat\Xi$, let $C^*(\xi)$ be the C*-algebra obtained from completing $\xi$ and extending operations.  Let $\vartheta_n:\mathcal M\to\bbU$ be as in Lemma \ref{l.unifhom}. To each $\xi\in\hat\Xi$ we have the associated $d_\xi\in\mathcal M$ defined by $d_\xi(n,m)=\|n-m\|_\xi$, and the map $\xi\mapsto d_\xi$ is clearly Borel. It is then straightforward to define $\bbX^\xi=(X^\xi,F_0^\xi,F^{\xi}_+,F^\xi_{\cdot},F^\xi_{*},(F^\xi_{q})_{q\in\bbQ(i)})\in M_{C^*}$ directly from $d_\xi$ and the $\vartheta_n$ so that $\bbX^\xi$ is isomorphic to $C^*(\xi)$, letting $F_0=\{\vartheta_{0_\xi}(d_\xi)\}$,
$$
F^\xi_+=\{(\vartheta_n(d_\xi),\vartheta_m(d_\xi),\vartheta_k(d_\xi)): n+_{\xi}m=k\},
$$
and defining $F^\xi_{\cdot}$, $F^\xi_{*}$, $(F^\xi_q)_{q\in\bbQ(i)}$ analogously. The map $\xi\mapsto\bbX^\xi$ is then Borel by the Kuratowski--Ryll-Nardzewski theorem.

For the converse direction, recall from \cite[\S 2.4]{FaToTo:Turbulence} that $\Xi\subseteq\bbR^\bbN$ consists of all real sequences $\eta$ such that for some C*-algebra $A$ and some $y=(y_n)_{n\in\bbN}$ which is dense in $A$, we have that $\eta_n=\|\mathfrak p_n(y)\|_A$, where $(\mathfrak p_n)_{n\in\bbN}$ enumerates the non-commutative $\bbQ(i)$-$*$-polynomials. Let $f_n:F(\bbU)\setminus\{\EMPTY\}\to\bbU$, $n\in \bbN$, be Borel functions provided by the Kuratowski--Ryll-Nardzewski theorem such that $\{f_n(X): n\in \bbN\}$ is a dense subset of $X$, for all $X\in F(\bbU)\setminus\{\EMPTY\}$. For $\bbX\in M_{C^*}$, let $\eta^X_n=\delta(\mathfrak p_n((f_n(X)_{n\in\bbN}),0_\bbX)$ where $F^\bbX_0=\{0_\bbX\}$, and $\mathfrak p_n((f_n(X))$ is the evaluation of $\mathfrak p_n$ at $y=(f_n(X))$ in the C*-algebra coded by $\bbX$. It is easily seen that $X\mapsto \eta^X$ is Borel and $\bbX$ and $\eta^\bbX$ code isomorphic C*-algebras.

\subsection{Banach spaces and Banach algebras}
A separable (real or complex) Banach space $E$ can be cast as a Polish structure $(E,F_0,F_+,(F_q)_{q\in\bbK})$, where $\bbK=\bbQ$ or $\bbK=\bbQ(i)$, in the signature $\mathcal L_{B}=(1,3,2,2,\ldots)$ in analogy with the above parametrization of C*-algebras (omitting multiplication and involution). The subset $M_B(\bbU)$ of $M(\mathcal L_{B},\bbU)$ corresponding to Banach spaces is easily seen to be Borel, too. By an argument similar to that in \S\ref{S.C*}, one sees that this parameterization is equivalent to the standard parameterization of Banach spaces as closed subspaces of $C([0,1])$ (see e.g., \cite{FeLouRo} or \cite{Mell:Computing}).

In a similar vein, letting $\mathcal L_{BA}=(1,3,3,2,2,\ldots)$ (which happens to coincide with $\mathcal L_{C^*}$), we can parametrize separable Banach algebras, as well as involutive separable Banach algebras, by appropriate Borel subsets of $M(\mathcal L_{BA})$.

In all cases, Theorem \ref{t.bga} provides that the isometric 
isomorphism relation is below a group action.

\subsection{Operator spaces and operator systems}\label{S.OS} To handle operator spaces and operator systems we will need a framework slightly more general than that of Polish structures. Formally, this could be handled by introducing \emph{multi-sorted} Polish structures, though here we make due with an ad hoc approach.

For metric spaces $(X,d_X)$ and $(Y,d_Y)$, let $UC(X,Y)$ denote the set of uniformly continuous functions $f:X\to Y$. Identifying each $f\in UC(X,Y)$ with $\graph(f)\subseteq X\times Y$, it is easily checked that $UC(X,Y)$ forms a Borel subset of $F(X\times Y)$.

Let $\bbU_n$, $n\in\bbN$, be a sequence of disjoint copies of the Urysohn metric space, with $\bbU_1=\bbU$. Identify $M_n(\bbU)$ with $\bbU^{n^2}$, and give this the supremum metric. 
We define $\bbOS$ to be the set of sequences $\bbE=(E_n,f_n)_{n\in\bbN}$ such that:
\begin{enumerate}
\item $(E_n,f_n)\in M_B(\bbU_n)\times UC(M_n(\bbU),\bbU_n)$;
\item $f_n(M_n(E_1))= \dom(E_n)$.
\item $f_n\restrict M_n(E_1)$ is linear;
\item for all $A,B\in M_n(\bbC)$ and all $X\in M_n(E_1)$ we have 
$$
\|f_n(AXB)\|_{E_n}\leq\|A\|\|f_n(X)\|_{E_n}\|B\|;
$$
\item for all $m,n\in\bbN$ and all $X\in M_n(E_1)$, $Y\in M_m(E_1)$
$$
\|X\oplus Y\|_{E_{n+m}}=\max\{\|X\|_{E_n},\|Y\|_{E_m}\|\}.
$$
\end{enumerate}
Once again, $\bbOS$ forms a Borel set. Each $\bbE\in\bbOS$ codes a $L^\infty$-matrix-normed space, whence by Ruan's theorem (e.g. \cite[Theorem 13.4]{Paulsen}), $\bbOS$ parametrizes the class of separable operator spaces. Defining $\bbE\simeq^{\bbOS}\bbF$ if and only if $\bbE$ and $\bbF$ are completely isometrically isomorphic, it is easy to check that $\simeq^{\bbOS}$ is analytic, and so $\bbOS$ provides a good standard Borel parametrization of the separable operator spaces. Also, write $\bbE\simeq^{\bbOS}_n\bbF$ if and only $\bbE$ and $\bbF$ are $n$-isometric.

The group $\Iso(\bbU)\times\Iso(\bbU_n)$ acts in a Borel way on\break
 $UC(M_n(\bbU),\bbU_n)$ by
$$
((\sigma,\tau)\cdot f)(x_{ij})=\tau(f(\sigma^{-1}(x_{ij})).
$$
Thus we obtain a Borel action of $\prod_n \Iso(\bbU_n)$ on $\bbOS$ by
$$
(\sigma_n)\cdot(E_n,f_n)=(\sigma_n\cdot E_n,(\sigma_1,\sigma_n)\cdot f_n),
$$
and we write $\bbE\equiv^{\bbOS}\bbF$ if and only there is $(\sigma_n)$ such that $(\sigma_n)\cdot \bbE=\bbF$. Arguing as we did in the proof of Theorem \ref{t.bga}, we obtain that $\simeq^{\bbOS}\leq_B\equiv^{\bbOS}$, which proves that $\simeq^{\bbOS}$ is below a group action.

If we only consider the action of $\prod_{j\leq n} \Iso(\bbU_n)$ on $\bbOS$, and denote by $\equiv^{\bbOS}_n$ the induced orbit equivalence relation, the argument from Theorem \ref{t.bga} gives that $\simeq^{\bbOS}_n\,\leq_B\,\equiv^{\bbOS}_n$, thus showing $\simeq^{\bbOS}_n$ is below a group action. This establishes Theorem \ref{T1}.(2).

\medskip

A standard parameterization of operator systems is obtained by adding the adjoint operation to the structures parametrized by $\bbOS$. Then an analogous argument establishes Theorem \ref{T1}.(3).

\medskip

By adding  a constant for the multiplicative unit to the language one  sees that 
the relation of unital complete isometry between operator systems is below a group action. 
Finally, unital maps between operator systems are completely isometric if and only 
if they are completely positive (see \cite[Proposition~3.6]{Paulsen}) and (4) and (5) of
Theorem~\ref{T1} follow.

\subsection{A remark about models of the logic for metric structures} \label{S.Met} 
The isomorphism relation of countable structures, in the sense of \cite{Hj:Book}, is given by a continuous $S_\infty$-action. 
The group $\Iso(\bbU)$ plays an analogous role for separable models of logic for metric structures. Such models consists of an underlying Polish space $X$, countably many functions $f_n\colon X^n\to X$, and countably many functions  $g_n\colon X^n\to \bbR$ (\emph{relations}). (Assuming  there are infinitely many functions and that the $n$'th function has $X^n$ as its domain is clearly not a loss of generality.) These functions are required to have a prescribed modulus of uniform continuity (see \cite{BYBHU}), 
but we shall ignore this since it is not important for our present purposes. 
To a fixed model $\mathcal X=(X,f_n, g_n: n\in \bbN)$, associate a Polish structure 
$\bbX=(X,(F_n)_{n\in \bbN})$ where $F_n$, $n\in \bbN$, enumerate 
the graphs of the $f_n$, as well as all the sets $\{\bar x\in X^n: g_n(x)\geq q_n\}$ where $(q_n)_{n\in\bbN}$ is a fixed enumeration of the rationals. This map is Borel (between the appropriate spaces) and $\mathcal X$ is isomorphic to $\mathcal Y$ if and only if $\bbX\simeq\bbY$. 
 Also, the map that sends an element of $\Met$ to a metric structure is  Borel. 

The following is proved by a straightforward recursion analogous to the case of classical logic. 

\begin{lemma} \label{L.2} 
If $T$ is a theory in a countable language in the logic of metric structures then 
the set of all $\bbX\in \Met$ that code a model of $T$ is 
Borel. \qed
\end{lemma} 

Combining this with Theorem \ref{t.bga}, we obtain:

\begin{theorem} If $T$ is a theory of the logic of metric structures in a separable 
language, then the isometry relation of models of $T$ 
is Borel-reducible to an orbit equivalence relation of $\Iso(\bbU)$. \qed
\end{theorem}

\section{Concluding remarks}

(1) Aaron Tikuisis has pointed out that the above approach, using Polish structures, also can be used to provide a new proof that the isomorphism relation for von Neumann algebras with separable predual is below a group action. Previously, in \cite{sato09a}, it was shown that isomorphism of separably acting von Neumann algebras is below an action of the unitary group of $\ell_2$, using a completely different line of argument. It is at present not known if the latter provides a sharper upper bound on complexity than the former.

(2) In \cite{FaToTo:Turbulence} it was proved that separable unital AI-algebras
are not classifiable by countable structures, and that their isomorphism relation is below an action of $\Aut(\mathcal O_2)$. However, we do not at present know if the complexity 
of the classification problem increases as one passes from nuclear to exact C*-algebras, or from exact to arbitrary C*-algebras. In particular, we do not know the answer to:

\begin{question}
Is the isomorphism relation for separable C*-algebras strictly more complicated, as measured by $\leq_B$, than that of nuclear separable C*-algebras?
\end{question}

Even if we restrict to nuclear \emph{simple} separable C*-algebras, we don't know the answer to this.

We may also ask if the upper bound on complexity provided by Theorem \ref{T1}.(1) is actually optimal. By \cite{GaKe}, the Borel actions of $\Iso(\bbU)$ realize the maximal complexity of equivalence relations induced by Polish group actions.

\begin{question}
Is the isomorphism relation for separable C*-algebras maximal among equivalence relations induced by a Polish group action?
\end{question}

We also don't know whether unital  $n$-order isomorphism of separable operator systems
is below a group action for~$n\in \bbN$.

\section{Acknowledgment}

The main result of this paper was proved in the aftermath of the BIRS 
workshop on applications of descriptive set theory to functional analysis. We would 
like to thank the staff at BIRS for providing an inspiring atmosphere.

G.A.E. and I.F.  partially supported by NSERC. 
C.R. was partially supported by NSF grants 1201295 and 0901405.
A.T.  supported in part by grant no. 10-082689 from the Danish Council for Independent Research.

\section*{Appendix}

We prove the generalization of the homogeneous selection principle, \cite[Lemma 6.2 and 6.3]{FaToTo:Turbulence}, which is used to prove Lemma \ref{l.unifhom} above. For a subset $A\subseteq X\times Y$ of a Cartesian product and $x\in X$, we define  $A_x=\{y\in Y: (x,y)\in A\}$. A function $f:\proj_X(A)\to Y$ is a \emph{uniformization} of $A$ if $(x,f(x))\in A$ for all $x\in \proj_X(A)$.

\begin{appthm}\label{t.hsp}
Let $X, Y$ be Polish spaces, $d_Y$ a complete compatible metric on $Y$, and let $A\subseteq X\times Y$ be a $G_\delta$ set.

(1) If for some (any) sequence $(y_n)_{n\in\bbN}$ dense in $Y$ the set
$$
R=\{(x,n,\varepsilon)\in X\times\bbN\times\bbQ_+: (\exists y\in A_x) d_Y(y,y_n)<\varepsilon\}
$$
is Borel, then $\proj_X(A)$ is Borel and $A$ admits a Borel uniformization.

(2) If $G$ is a Polish group acting continuously on $Y$ by Borel automorphisms such that $A_x$ is $G$-invariant for all $x\in X$, and every $G$-orbit of a point $y\in A_x$ is dense in $A_x$, then $R$ defined in (1) is Borel, and so $A$ admits a Borel uniformization.
\end{appthm}
\begin{proof}
(1) We may assume that $d_Y$ is bounded by $\frac 1 2$. Clearly $\proj_X(A)=\{x\in X: (\exists n) (x,n,1)\in R\}$, so this set is Borel. To construct the uniformization, fix open sets $U_n\subseteq X\times Y$ such that $\bigcap_{n\in\bbN} U_n=A$, and fix an enumeration $(q_n)_{n\in\bbN}$ of $\bbQ_+$. Let $B(y,\varepsilon)$ denote the open $d_Y$-ball of radius $\varepsilon$ around $y\in Y$. We recursively define Borel maps $x\mapsto n_i(x)\in\bbN$ and $x\mapsto \varepsilon_i(x)\in\bbQ_+$, $i\in\bbN$, on $\proj_X(A)$, satisfying:
\begin{enumerate}
\item $n_i(x)=0$, $\varepsilon_0(x)=1$;
\item $\varepsilon_i(x)\leq \frac 1 {2^i}$;
\item $B(y_{n_i(x)},\varepsilon_i(x))\subseteq (U_i)_x$;
\item $(x,n_i(x),\varepsilon_i(x))\in R$;
\item $\overline{B(y_{n_{i+1}(x)},\varepsilon_{i+1})}\subseteq B(y_{n_i(x)},\varepsilon_i(x))$.
\end{enumerate}
If this can be done then $f(x)=\lim_{i\to\infty} y_{n_i(x)}$ is the desired Borel uniformization. Suppose $n_i(x)$ and $\varepsilon_i(x)$ have been defined for $i\leq k$. Let 
$$
z\in B(y_{n_k(x)},\varepsilon_k(x))\cap A_x
$$
and let $0<\delta\leq 2^{k+1}$ be such that $2\delta+d_Y(z,y_{n_k(x)})<\varepsilon_k(x)$ and $B(z,2\delta)\subseteq U_{k+1}$. If $d(y_n,z)<\delta$, then $\overline{B(y_n,\delta)}\subseteq B(y_{n_k(x)},\varepsilon_k(x))$, $B(y_n,\delta)\cap A_x\neq\EMPTY$, and $B(y_n,\delta)\subseteq U_{k+1}$. This shows that (1)--(5) above can be satisfied, and so we can define $n_{k+1}(x)$ and $\varepsilon_{k+1}(x)=q_n$, where $n_{k+1}(x)$ and $n$ are least possible satisfying (1)--(5) above. Since each requirement (1)--(5) are Borel, the maps $x\mapsto n_{k+1}(x)$ and $x\mapsto \varepsilon_{k+1}(x)$ are Borel, as required.

(2) It is clear from the definition that $R$ is analytic. Since
$$
R=\{(x,n,\varepsilon): (\forall y\in A_x)(\exists g\in G_0) d(g\cdot y,y_n)<\varepsilon\}
$$
for a countable dense $G_0\subseteq G$
it is also coanalytic, whence $R$ is Borel.
\end{proof}

\bibliographystyle{amsplain}
\bibliography{isom-bga}

\end{document}